\documentclass[11pt]{article}
\usepackage{amsmath,amsfonts,amssymb,latexsym,amsthm,dsfont}
\usepackage{geometry,graphicx}
\usepackage[backref]{hyperref}

\title{Long time behavior of diffusions with Markov switching}

\author{%
  Jean-Baptiste~\textsc{Bardet}, %
  H\'el\`ene~\textsc{Gu\'erin}, %
  Florent~\textsc{Malrieu}}

\date{Preprint -- \today}

\newcommand{\ds}{\displaystyle}
\newcommand{\ind}{\mathds{1}}
\newtheorem{thm}{Theorem}[section]%
\newtheorem{prop}[thm]{Proposition}%
\newtheorem{rem}[thm]{Remark}%
\newtheorem{nota}[thm]{Notations}%
\newcommand{\dE}{\mathbb{E}}

\newcommand{\dP}{\mathbb{P}}
\newcommand{\dR}{\mathbb{R}}

\newcommand{\cA}{\mathcal{A}}
\newcommand{\cC}{\mathcal{C}}
\newcommand{\cE}{\mathcal{E}}\newcommand{\cF}{\mathcal{F}}

\newcommand{\cL}{\mathcal{L}}
\newcommand{\cN}{\mathcal{N}}
\newcommand{\cP}{\mathcal{P}}

\newcommand{\ABS}[1]{{{\left| #1 \right|}}} 
\newcommand{\BRA}[1]{{{\left\{#1\right\}}}} 
\newcommand{\PAR}[1]{{{\left(#1\right)}}} 
\newcommand{\SBRA}[1]{{{\left[#1\right]}}} 
\renewcommand{\leq}{\leqslant}
\renewcommand{\geq}{\geqslant}

\begin{document}

\maketitle

\begin{abstract}
  Let $Y$ be an Ornstein-Uhlenbeck diffusion governed by an ergodic finite state Markov process $X$: $dY_t=-\lambda(X_t)Y_tdt+\sigma(X_t)dB_t$, $Y_0$ given. Under ergodicity condition, we get quantitative estimates for the long time behavior of $Y$. We also establish a trichotomy for the tail of the stationary distribution of $Y$: it can be heavy (only some moments are finite), exponential-like (only some exponential moments are finite) or Gaussian-like (its Laplace transform is bounded below and above by Gaussian ones). The critical moments are characterized by the parameters of the model. 
\end{abstract}


\noindent\textbf{AMS Classification 2000: }60J60, 60J75, 60H25.


\section{Introduction and main results}

The aim of this paper is to draw a complete picture of the ergodicity of Ornstein-Uhlenbeck diffusions with Markov switching (characterization of the tails of the invariant measure and quantitative convergence to equilibrium). In particular we make more precise the results of \cite{guyon,saporta-yao}. The so-called diffusion with Markov switching $Y={(Y_t)}_{t\geq 0}$ is defined as follows.

The \emph{switching process} $X={(X_t)}_{t\geq 0}$ is a Markov process
on the finite state space $E=\BRA{1,\ldots,d}$ (with $d\geq 2$), of infinitesimal
generator $A={(A(x,\tilde x))}_{x,\tilde x\in E}$. Let us denote by $a(x)$ the jump
rate at state $x\in E$ and $P={(P(x,\tilde x))}_{x,\tilde x\in E}$ the transition
matrix of the embedded chain. One has, for $x\neq \tilde x$ in
$E$,
$$
a(x)=-A(x,x)%
\quad\text{and}\quad P(x,\tilde x)=-\frac{A(x,\tilde x)}{A(x,x)}.
$$
We assume that $P$ is irreducible recurrent. The process $X$ is ergodic with a unique
invariant probability measure denoted by $\mu$. See \cite{norris} for
details.  Let $\cF_t^X=\sigma(X_u,0\leq u\leq t)$. Moreover, let $\dE_x$ denote 
the expectation with respect to the law
$\dP_x$ of $X$ knowing that $X_0=x$. 

Let $B={(B_t)}_{t\geq 0}$ be a standard Brownian motion on $\dR$ and
$Y_0$ a real-valued random variable such that $B$, $Y_0$ and $X$ are
independent. Conditionnally to $X$, the process $Y={(Y_t)}_{t\geq 0}$ is the
real-valued diffusion process defined by:
\begin{equation}
  \label{eq:defY}
  Y_t=Y_0-\int_0^t\!\lambda(X_u)Y_u\,du+\int_0^t\sigma(X_u)\,dB_u,
\end{equation}
where $\lambda$ and $\sigma$ are two functions from $E$ to $\dR$ and $(0,\infty)$ respectively. Of
course, if $\lambda$ and $\sigma$ are constant, $Y$ is just an Ornstein-Uhlenbeck process with attractive ($\lambda>0$), neutral ($\lambda=0$) or
repulsive coefficient ($\lambda<0$). One has to notice that Equation \eqref{eq:defY} 
has an ``explicit'' solution: 
\begin{equation}
  \label{eq:explicit}
Y_t=Y_0\exp\PAR{-\int_0^t\!\lambda(X_u)\,du}+%
\int_0^t\exp\PAR{-\int_u^t\!\lambda(X_v)\,dv}\sigma(X_u)\,dB_u.
\end{equation}
\begin{rem}
 In others words, the full process $(X,Y)$ is the Markov process on $E\times\dR$ associated to the infinitesimal generator $\cA$ defined by: 
$$
\cA f(x,y)=\sum_{\tilde x\in E}A(x,\tilde x)(f(\tilde x,y)-f(x,y))%
+\frac{\sigma(x)^2}{2}\partial^2_{22}f(x,y)-\lambda(x)\partial_2f(x,y). 
$$
\end{rem}
Previous works investigated the ergodicity of $Y$ and some integrability properties 
for the invariant measure. For example, in \cite{basak}, the multidimensional case 
is adressed together with the case of diffusion coefficients depending on $Y$. 
Stability results and sufficient conditions for the existence of moments are established 
under Lyapunov-type conditions. 

In \cite{guyon}, it is proved that the Markov switching diffusion $Y$
is ergodic if and only if
\begin{equation}
  \label{eq:existence}
  \sum_{x\in E}\lambda(x)\mu(x)>0,
\end{equation}
that is if the process is attractive ``in average''. Let us denote by $\nu$ 
its invariant probability measure of $Y$. It is also shown in \cite{guyon}
that $\nu$ admits a moment of order $p$ if, for any $x\in E$,  $p\lambda(x)+a(x)$ is positive 
and the spectral radius of the matrix 
\begin{equation}\label{eq:Ms}
M_p=\PAR{\frac{a(x)}{a(x)+p\lambda(x)}P(x,\tilde x)}_{x,\tilde x\in E}
\end{equation}
is smaller than 1. In the sequel $\rho(M)$ stands for the spectral 
radius of a matrix $M$. 

In \cite{saporta-yao}, the result is more precise: a dichotomy is exhibited between 
heavy and light tails for $\nu$. Let us define
\begin{equation}\label{eq:minl}
\underline \lambda=\min_{x\in E}\lambda(x)%
\quad\text{and}\quad%
\overline \lambda=\max_{x\in E}\lambda(x).
\end{equation}
\begin{thm}[de~Saporta-Yao \cite{saporta-yao}]\label{th:kappa}
Under Assumption \eqref{eq:existence}, the following dichotomy holds: 
\begin{enumerate}
\item if $\underline\lambda <0$, then there exists $C>0$ such that
$$
t^\kappa\nu((t,+\infty))\xrightarrow[t\rightarrow+\infty]{}C,
$$
where $\kappa$ is the unique 
$p\in (0, \min\BRA{-a(x)/\lambda(x),\ \lambda(x)<0})$ 
such that the spectral radius of $M_{p}$ is equal to 1;
\item if $\underline\lambda\geq 0$, then $\nu$ has moments of all order.  
\end{enumerate}
\end{thm}

\begin{rem}
 Note that the constant $\kappa$ does not depend on the parameters ${(\sigma(x))}_{x\in E}$, and that Point 1. from previous theorem implies that, for $\underline \lambda<0$, the $p^\text{th}$ moment of $\nu$ is finite if and only if $p<\kappa$.\\
The main idea of the proofs in  \cite{guyon} and \cite{saporta-yao}
 is to study the discrete time Markov chain 
${(X_{\delta n},Y_{\delta n})}_{n\geq 0}$ for any $\delta>0$ with 
renewal theory and then to let $\delta$ goes to 0. 
\end{rem}

The main goal of the present paper is to show that there are three (and not only two) 
different behaviors for the tails of $\nu$. 

Let us gather below several useful notations. 
\begin{nota}\label{def:Ap} Let us define for the diffusion coefficients
\begin{equation}\label{eq:mins}
\underline \sigma^2=\min_{x\in E}\sigma^2(x)%
\quad\text{and}\quad%
\overline \sigma^2=\max_{x\in E}\sigma^2(x).
\end{equation}

We denote by $A_p$ the matrix $A-p\Lambda$ where $\Lambda$ is the diagonal matrix with diagonal
$(\lambda(1),\ldots,\lambda(d))$ and associate to $A_p$ the quantity
\begin{equation}\label{eq:eta}
\eta_p:=-\max_{\gamma\in\mathrm{Spec}(A_p)} \mathrm{Re\ }\gamma.
\end{equation}
When $\underline\lambda\geq 0$, the set $E$ is the union of 
\begin{equation}\label{eq:MN}
M=\BRA{x\in E,\ \lambda(x)>0}%
\quad\text{and} \quad%
 N=\BRA{x\in E,\ \lambda(x)=0}. 
\end{equation}
Let us then define
\begin{equation}\label{eq:beta}
\beta(x)=\frac{\sigma(x)^2}{2a(x)}%
\quad\text{and}\quad%
\overline{\beta}=\max_{x\in N}\beta(x), 
\end{equation}
and, for any $v$ such that $v^2<\overline\beta^{-1}$, the matrix %
\begin{equation}\label{eq:PN}
P^{(N)}_v= \PAR{\frac{1}{1-\beta(x)v^2}P(x,x')}_{x,x'\in N}.
\end{equation}
\end{nota}

We are now able to state our main result.

\begin{thm}\label{th:tricho}
Let us define 
$$
\kappa=\sup\BRA{p\geq 0,\ \eta_p>0}\in (0,+\infty].
$$
Then $\eta_p$ is continuous, positive on the set $(0,\kappa)$ and negative on $(\kappa,+\infty)$. 
Under Assumption \eqref{eq:existence}, the following trichotomy holds: 
\begin{enumerate}
\item if $\underline\lambda<0$ then $0<\kappa\leq  \min\BRA{-a(x)/\lambda(x),\ \lambda(x)<0}$,
and the $p^{th}$ moment of $\nu$ is finite if and only if $p<\kappa$.
\item if $\underline\lambda=0$, then $\kappa$ is infinite and the domain of the Laplace transform of $\nu$
is $(-v_c,v_c)$ where 
\begin{equation}\label{eq:vc}
 v_c=\sup\BRA{v>0,\ \rho(P^{(N)}_v)<1};
\end{equation}
\item if $\underline\lambda> 0$, then $\kappa$ is infinite and $\nu$ has a Gaussian-like Laplace transform: for any $v\in\dR$, 
  $$
  \exp\PAR{\frac{\underline\sigma^2v^2}{4\overline\lambda}} \leq
  \int\!e^{vy}\,\nu(dy)\leq%
  \exp\PAR{\frac{\overline\sigma^2v^2}{4\underline\lambda}}.
  $$  
Moreover, its tail looks like the one of the Gaussian law with variance $\overline \alpha/2$ where $\overline \alpha=\max_{x\in E}\sigma(x)^2/\lambda(x)$ since $y\mapsto e^{\delta y^2}$ is $\nu$-integrable if and only if $\delta<1/\overline\alpha$. 
\end{enumerate}
\end{thm}


\begin{rem}
In the sequel we will respectively refer to Points 1. 2. and 3. as the polynomial, exponential-like and Gaussian-like cases.  
\end{rem}

The first point of this theorem is a reformulation of the first point of Theorem \ref{th:kappa} by de~Saporta and Yao.  We can in particular check that our characterization of $\kappa$ in Theorem \ref{th:tricho} is equivalent to the one given by de~Saporta and Yao in Point 1. of Theorem \ref{th:kappa} (see Remark \ref{rem:kappa}). We provide a direct and simple proof of this result based on It\^o formula and some basic results on finite Markov chains. The proof of Points 2. relies on precise estimates on the Laplace transform of $Y_t$ that can be derived from a discrete time model already studied in \cite{goldie-grubel,hitwes,iksanov}. 

It is straightforward from \eqref{eq:explicit} that, for any measure $\pi_0$ on
$E\times\dR$, the Laplace transform $L_t$ of $Y_t$ is 
\begin{equation}\label{eq:laplaceYt}
L_t(v):=\dE_{\pi_0}\PAR{e^{v Y_t}}=%
\dE_{\pi_0}\SBRA{\exp\PAR{vY_0e^{-\int_0^t\!\lambda(X_s)\,ds}%
    +\frac{v^2}{2}\int_0^t\!\sigma(X_s)^2%
    e^{-2\int_s^t\!\lambda(X_r)\,dr}\,ds}}. 
 \end{equation}
 
The estimate of the Laplace transform in the Gaussian-like case (Point 3.) is hence 
easily deduced from this explicit expression.  Assuming that $Y_0=0$, we get from 
\eqref{eq:laplaceYt} that 
\begin{equation*} 
L_t(v) \leq \dE\SBRA{\exp\PAR{\frac{v^2}{2}\int_0^t\!\overline \sigma^2%
    e^{-2\int_s^t\!\underline\lambda\,dr}\,ds}}
    \leq  
    \exp\PAR{\PAR{1-e^{-2\underline\lambda t}}\frac{\overline\sigma^2v^2}{4\underline\lambda}},
\end{equation*}
which gives the upper bound as $t$ goes to infinity. The lower bound follows from a symmetric argument.

The proofs of  Point 2. and of the second part of Point 3.  are more delicate (and interesting).  For the exponential case, we first get the critical exponential moment for the process $Y$ observed at the hitting times of the subset $M$ defined in \eqref{eq:MN}. Then we show that the full process has the same critical exponent. 

At the end of the paper we focus on the convergence of the law of $Y_t$ to the invariant measure $\nu$. We get an explicit exponential bound for the Wasserstein distance of order $p$ for any $p<\kappa$. Classically, let $p\geq 1$ and $\cP_p$ be the set of the probability measures on $\dR$ with a finite $p^{th}$ moment. Define the Wasserstein distance $W_p$ on $\cP_p$ as follows: for any $\rho$ and $\tilde \rho$ in $\cP_p$, 
$$
W_p(\rho,\tilde \rho)=\PAR{\inf_\pi \BRA{\int\!\ABS{y-\tilde y}^p\,\pi(dy,d\tilde y)}}^{1/p}, 
$$
where the infimimum is taken among all  the probability measures $\pi$ on $\dR^2$ with marginals $\rho$ and $\tilde \rho$. It is well-known that $(\cP_p,W_p)$ is a complete metric space (see \cite{villani}). 

The strategy is to couple two processes $(X,Y)$ and $(\tilde X,\tilde Y)$ in such a way that the Wasserstein distance between $\cL(Y_t)$ and $\cL(\tilde Y_t)$ goes to zero as $t$ goes to infinity. This requires to couple the initial conditions and the dynamics (of both the Markov chains and the diffusion part). When $X_0$ and $\tilde X_0$ have the same law, the coupling is trivial:  we choose $X=\tilde X$ and the same driving Brownian motion. 

\begin{thm}\label{th:Xergo}
  Let $p<\kappa$.  Assume that $X_0$ and $\tilde X_0$ have the same
  law. Let $Y$ and $\tilde Y$ be solutions of \eqref{eq:defY} associated to $(X_{t})$ and $(\tilde{X}_{t})$ and
  assume that $Y_0$ and $\tilde Y_0$ have finite moment of order
  $p$. Then there exists $C(p)$ such that 
  $$
  W_p\PAR{\cL(Y_t),\cL(\tilde Y_t)}^p\leq%
  C(p)e^{-\eta_p t}W_p\PAR{\cL(Y_0),\cL(\tilde Y_0)}^p  ,
  $$
  where $\eta_p$ is given by \eqref{eq:eta}.
  \end{thm}

If $X_0$ and $\tilde X_0$ do not have the same law, one first has to make the Markov chains $X$ and $\tilde X$ stick together and then to use Theorem \ref{th:Xergo}. This provides a rather intricate bound which is given for convenience in Section \ref{sec:moment-convergence}. 

The paper is organised as follows. In Section \ref{sec:gauss-mom} we complete the proof for the Gaussian-like case of Theorem \ref{th:tricho}. The exponential-like case is studied in Section \ref{sec:expo-mom}.  Since the critical exponential moment is not explicit in the general case, we give also  the explicit computation of the Laplace transform of $\nu$ when $E$ is reduced to $\BRA{1,2}$.  In Section \ref{sec:yaothereturn} we establish a uniform bound for the $p^{th}$ moment of ${(Y_t)}_t$ for any $p<\kappa$ and the first point of Theorem \ref{th:tricho} as a corollary. We finally provide the proof of Theorem \ref{th:Xergo} and its extension to general initial conditions  in Section \ref{sec:moment-convergence}.  

\section{Gaussian moments for the switched diffusion} \label{sec:gauss-mom}

This section is dedicated to the proof of the second part of Point 3. of Theorem \ref{th:tricho}. 

\begin{proof}[Proof of Point 3. of Theorem \ref{th:tricho}]
Let us denote by 
$$
\alpha(x)=\frac{\sigma(x)^2}{\lambda(x)}%
\quad\text{for $x\in E$ and}\quad
\overline\alpha=\max_{x\in E}\alpha(x)<+\infty. 
$$
For any $\delta\in(0,1/\overline \alpha)$, It\^o's formula ensures that 
\begin{align*}
d e^{\delta Y_t^2}&=%
\PAR{-2\lambda(X_t)\delta Y_t^2+(2\delta^2Y_t^2+\delta)\sigma(X_t)^2}e^{\delta Y_t^2}dt+dM_t
\end{align*}
where ${(M_t)}_t$ is a martingale. For any $x\in E$ and $y\in\dR$, 
\begin{align*}
2(-\lambda(x)+\delta\sigma(x)^2)y^2+\sigma(x)^2&\leq 
-2\lambda(x)(1-\delta \overline\alpha )y^2+\overline\alpha \lambda(x)\\
&\leq -2\underline\lambda(1-\delta \overline\alpha )y^2+\overline\alpha \overline\lambda,
\end{align*}
since $\delta\overline\alpha<1$. Moreover, for any $a>0$, there exists $b>0$ such that, for any $y\in\dR$,  
$$
-2\underline \lambda \delta (1-\overline\alpha\delta)y_t^2+\overline\lambda\overline\alpha\delta
\leq -a+be^{-\delta y^2}, 
$$
thus 
$$
\frac{d}{dt}\dE\PAR{e^{\delta Y_t^2}}\leq %
-a \dE\PAR{e^{\delta Y_t^2}}+b.
$$
As a consequence, $\sup_{t\geq 0}\dE\PAR{e^{\delta Y_t^2}}$ is finite as soon as $\dE\PAR{e^{\delta Y_0^2}}$ is finite and $\delta\overline\alpha<1$. 

On the other hand, assume (without loss of generality) that $\alpha(1)=\overline\alpha$.  Choose $(X_0,Y_0)$ with law $\overline\nu$ (the invariant measure of $(X,Y)$). For any $t>0$,  we have 
$$
\dE\PAR{e^{\delta Y_0^2}}=\dE\PAR{e^{\delta Y_t^2}}
\geq \dE\SBRA{\ind_\BRA{X_0=1}\dE_{1,Y_0}\PAR{\ind_\BRA{T_1>t}e^{\delta Y_t^2}}},
$$
where $T_1$ is the first jump time of $X$.  On the set $\BRA{T_1>t}$, 
$$
Y_t\overset{\cL}{=}Y_0 e^{-\lambda(1)t}+N_t
$$
where $N_t$ is a centered Gaussian random variable with variance $\alpha(1)(1-e^{-2\lambda(1)t})/2$ which is independent of $Y_0$ and $T_1$. Thus, reminding that $T_1\sim\cE(a(1))$, we get
$$
\dE_{1,Y_0}\PAR{\ind_\BRA{T_1>t}e^{\delta Y_t^2}}=%
e^{-a(1)t}\dE\PAR{e^{\delta(Y_0 e^{-\lambda(1)t}+N_t)^2}}.
$$
Since $a\mapsto \dE\PAR{e^{\delta(a+N_t)^2}}$ is even and convex, it reaches its minimum at $a=0$ and
$$
\dE\PAR{e^{\delta(Y_0 e^{-\lambda(1)t}+N_t)^2}}\geq \dE\PAR{e^{\delta N_t^2}}=
\begin{cases}
 \displaystyle{\frac{1}{\sqrt{1-\delta\alpha(1)(1-e^{-2\lambda(1)t})}}}&\text{ if }\delta\alpha(1)(1-e^{-2\lambda(1)t})<1,\\
 +\infty&\text{ otherwise.}
\end{cases}
$$
As a consequence, if $\delta>1/\alpha(1)$, $\dE\PAR{e^{\delta Y_t^2}}$ is bounded below by a function of $t$ which is infinite for $t$ large enough. Thus, $\dE\PAR{e^{\delta Y_t^2}}$ is infinite too.  
\end{proof}

\section{Exponential moments for the switched diffusion} \label{sec:expo-mom}


This section is dedicated to the proof of Point 2. in Theorem \ref{th:tricho}. We assume in the sequel that $\underline\lambda=0$. If ${(X_{t})}_{t\geq0}$ is a two-states Markov process then one can use \eqref{eq:laplaceYt} to compute explicitely the Laplace transform of the invariant measure $\nu$. This is a warm-up for the general case, and gives a more explicit formula for the critical exponential moment, whereas it will come from an abstract spectral criterion in the general case. 

\subsection{The explicit expression for the two-states case}

In this subsection we assume that $E=\BRA{1,2}$ and that $\underline\lambda=0$. Let us start with a straightforward computation which suggests that the Laplace transform of the invariant measure of $Y$ 
is infinite outside a bounded interval. 

\begin{rem}
  If $T$ is an exponential random variable with parameter $a$ and $B$
  is a standard Brownian motion on $\dR$ (with $T$ and $B$
  independent) then,
  $$
  \dE\PAR{e^{v \sigma B_T}} %
  =\int_0^\infty \dE\PAR{e^{v \sigma B_t}}ae^{-at}\,dt =\int_0^\infty
  e^{\sigma^2 v^2 t/2 }ae^{-at}\,dt
  =\frac{2a}{2a-\sigma^2v^2}.
  $$
  In other words, the law of $\sigma B_T$ is a (symmetric) Laplace law. When $X$ spends an exponential time in $x\in E$ with $\lambda(x)=0$, $Y$ behaves like $\sigma(x) B$. 
\end{rem}

\begin{thm}[The two-states degenerate case]
  Assume that $E=\BRA{1,2}$, $\lambda(1)=\lambda>0$ and
  $\lambda(2)=0$. Then, for any $v$ such that  $v^2<1/\beta(2)$ 
  (see \eqref{eq:beta} for the definition of $\beta$),
\begin{equation}\label{eq:Laplacedeux}
L(v)=\int_{-\infty}^{+\infty}e^{vx}\nu(dx)=%
\PAR{\frac{1-\mu(1)\beta(2)v^2}{1-\beta(2)v^2}}%
\PAR{\frac{1}{1-\beta(2)v^2}}^{1+a(1)/\lambda}%
\exp\PAR{\frac{\sigma(1)^2v^2}{4\lambda}}.
\end{equation}
If $v^2\geq 1/\beta(2)$, $L(v)$ is infinite. 
\end{thm}

\begin{proof}
Since $E=\BRA{1,2}$, $X$ is symmetric with respect to $\mu$ which is given by $\mu(1)=a(2)/(a(1)+a(2))$. Let us denote by $L_t$ the Laplace transform of $Y_t$ when $Y_0=0$ and $X$ is stationnary \emph{i.e.}  
$\cL(X_0)=\mu$. From Equation \eqref{eq:laplaceYt}, one has for any $v\in \dR$, 
\begin{align*}
L_t(v)&=%
\dE_{\mu}\SBRA{\exp\PAR{\frac{v^2}{2}\int_0^t\!\sigma(X_s)^2%
    e^{-2\int_s^t\!\lambda(X_r)\,dr}\,ds}}\\
&=    \dE_{\mu}\SBRA{\exp\PAR{\frac{v^2}{2}\int_0^t\!\sigma(X_s)^2%
    e^{-2\int_0^s\!\lambda(X_r)\,dr}\,ds}}
\end{align*}
since $\mu$ is reversible. By monotone convergence, we get that, for any $v\in\dR$, 
$$
L(v)=    \dE_{\mu}\SBRA{\exp\PAR{\frac{v^2}{2}\int_0^\infty\!\sigma(X_s)^2%
    e^{-2\int_0^s\!\lambda(X_r)\,dr}\,ds}}\in [1,+\infty],
    $$
where $L$ is the Laplace transform of $\nu$. 

  Let us introduce two auxilliary functions: for $x=1,2$,
  $$
  L_x(v)= \dE_x\SBRA{\exp\PAR{\frac{v^2}{2}%
      \int_0^\infty\!\sigma(X_s)^2e^{-2\int_0^s\!\lambda(X_r)\,dr}\,ds}}.
  $$
  It is clear that
  $$
  L(v)=\mu(1)L_1(v)+\mu(2)L_2(v).
  $$
  Moreover, if for any $t\geq 0$, $\cF_t=\sigma(X_s,\ 0\leq s\leq t)$
  and $T$ is the first jump time of $X$, then
  \begin{align*}
    L_x(v)&=\dE_x\SBRA{\dE_x\BRA{\exp\PAR{\frac{v^2}{2}%
          \int_0^\infty\!\sigma(X_s)^2e^{-2\int_0^s\!\lambda(X_r)\,dr}\,ds}\Big|\cF_T}}\\
    &=\dE_x\SBRA{\exp\PAR{\frac{v^2}{2}
        \int_0^{T}\!\sigma(X_s)^2e^{-2\int_0^s\!\lambda(X_r)\,dr}\,ds}%
E_{x,T}},
  \end{align*}
  where 
  $$
  E_{x,T}= \dE_x\BRA{\exp\PAR{\frac{v^2}{2}\int_{T}^\infty\!%
          \sigma(X_s)^2e^{-2\int_0^s\!\lambda(X_r)\,dr}\,ds}\Big|\cF_{T}}.
  $$
  For any $s\in [0,T[$, $X_s=x$ and then
  $$
  \int_0^{T}\!\sigma(X_s)^2e^{-2\int_0^s\!\lambda(X_r)\,dr}\,ds%
  =\sigma(x)^2\frac{1-e^{-2\lambda(x)T}}{2\lambda(x)},
  $$
  with the convention $(1-e^{-0\times T})/0=T$. Similarly, for $t\geq
  T$,
  $$
  \int_{T}^\infty\!\sigma(X_s)^2%
  e^{-2\int_0^s\!\lambda(X_r)\,dr}\,ds%
  = e^{-2\lambda(x)T}%
  \int_{T}^\infty\!\sigma(X_s)^2e^{-2\int_T^s\!\lambda(X_r)\,dr}\,ds
  $$
  The Markov property implies
  $$
  \dE_x\SBRA{\exp\PAR{\frac{v^2}{2}\int_{T}^\infty\!%
      \sigma(X_s)^2e^{-2\int_0^s\!\lambda(X_r)\,dr}\,ds}\Big|\cF_{T}}%
  =L_{X_T}\PAR{ve^{-\lambda(x)T}}.
  $$
  Thus,
  $$
  L_x(v)=\dE\SBRA{\exp\PAR{\frac{v^2\sigma(x)^2(1-e^{-2\lambda(x)T})}%
      {4\lambda(x)}} L_{3-x}\PAR{v e^{-\lambda(x)T}}\Big|X_0=x}.
  $$
  More precisely,
  $$
  L_1(v)=\dE_1\SBRA{\exp\PAR{\frac{v^2\sigma(1)^2(1-e^{-2\lambda
          T})}{4\lambda}}%
    L_2\PAR{v e^{-\lambda T}}},
  $$
  and
  $$
  L_2(v)=%
  \dE_2\SBRA{e^{v^2\sigma(2)^2T/2}L_1(v)}%
  =
  \begin{cases}
    \displaystyle{\frac{2a(2)}{2a(2)-\sigma(2)^2v^2}L_1(v)}&\text{ if
    }\sigma(2)^2v^2<2a(2),\\
    +\infty&\text{ otherwise.}
  \end{cases}
  $$
  Using $\beta(2)=\sigma(2)^2/2a(2)$, one easily gets that $L_1$ satisfies the following equation: for
  any $v^2<1/\beta(2)$,
  \begin{align*}
    L_1(v)&=%
    \frac{1}{1-\beta(2)v^2}%
    \int_0^\infty\!\exp\PAR{%
      \frac{\sigma(1)^2v^2(1-e^{-2\lambda t})}{4\lambda}}%
    L_1(v e^{-\lambda t})a(1)e^{-a(1)t}\,dt\\
    &=\frac{1}{1-\beta(2)v^2}%
    \int_0^1\!\exp\PAR{\frac{\sigma(1)^2v^2(1-u^2)}{4\lambda}}%
    L_1(v u)\frac{a(1)}{\lambda}u^{a(1)/\lambda-1}\,du.\\
\end{align*}
With $x=uv$, 
$$
L_1(v)=%
\frac{1}{1-\beta(2)v^2}%
\PAR{\frac{1}{v}}^{a(1)/\lambda}e^{\sigma(1)^2v^2/(4\lambda)}%
\int_0^v\!e^{-\sigma(1)^2x^2/(4\lambda)}\frac{a(1)}{\lambda}%
x^{a(1)/\lambda-1}L_1(x)\,dx.
$$
Deriving this relation provides 
$$
L_1'(v)=\PAR{\frac{\beta(2)v}{1-\beta(2)v^2}-\frac{a(1)}{\lambda v}%
+\frac{\sigma(1)^2v}{2\lambda}+%
\frac{1}{1-\beta(2)v^2}\frac{a(1)}{\lambda v}}L_1(v). 
$$
Then $L_1$ is solution of 
$$
L_1'(v)=\PAR{\frac{\sigma(1)^2v}{2\lambda}%
+\frac{\beta(2)(1+a(1)/\lambda)v}{1-\beta(2)v^2} }L_1(v)
$$
which leads to 
$$
L_1(v)=e^{\sigma(1)^2v^2/(4\lambda)}%
\PAR{\frac{1}{1-\beta(2)v^2}}^{1+a(1)/\lambda},
$$
since $L_1(0)=1$. Since $L_2$ is a function of $L_1$ we get 
$$
L(v)=e^{\sigma(1)^2v^2/(4\lambda)}%
\PAR{\frac{1-\mu(1)\beta(2)v^2}{1-\beta(2)v^2}}%
\PAR{\frac{1}{1-\beta(2)v^2}}^{1+a(1)/\lambda}.
$$
\end{proof}

\subsection{The exponential-like case}

In this subsection we provide the proof of Point 2. ($\underline\lambda=0)$ of Theorem \ref{th:tricho}. We first recall that, in this case, we split the state space $E$ of the switching process $X$ in two subsets $M$ and $N$ defined in \eqref{eq:MN}. We denote also by $F$ the points of $M$ that can be reached in one step from $N$:
$$
F=\BRA{x\in M,\ \sum_{\tilde x\in N}P(\tilde x,x)>0}.
$$
Assume for simplicity that $X_0\in M$ and define by induction the sequence of times ${(T_n)}_{n\geq 0}$ by $T_0=0$ and, for $n\geq 0$, 
$$
T_{2n+1}=\inf\BRA{t>T_{2n},\ X_t\in N},%
\quad\text{and}\quad%
T_{2n+2}=\inf\BRA{t>T_{2n+1},\ X_t\in M}.
$$
When $X$ is in $M$, $Y$ looks like a Ornstein-Uhlenbeck process (with variable but attractive drift) while it looks like a Brownian motion (with variable but bounded below and above variance) when $X$ is in $N$. Thus, heuristically the process $Y$ might be larger after a sojourn of $X$ in $N$ than in $M$.  

Let us notice that for $x\in N$, 
$$
Y_T=Y_0+I_x
\quad\text{where}\quad
I_x=\int_0^T\!\sigma(X^x_s)\,dB_s
$$
and $X^x$ is the process $X$ starting at $x$ and $T$ is the first hitting time of $M$.
Our strategy is to determine the domain of the Laplace transform of $I_x$ and then to 
establish that is also the one of the process $Y$ at the entrance times of $X$ into 
the set $M$ \emph{i.e} at the times ${(T_{2n})}_{n\geq 0}$. We will then extend 
the result to the full process $(X,Y)$.

%

\begin{prop}\label{prop:spectral}
 Under previous assumptions, for any
$v^2<\overline{\beta}^{-1}$, the two following conditions are equivalent:
  \begin{enumerate}
  \item for any $x\in N$, $\dE(e^{v I_x})<+\infty$;
  \item $\rho(P^{(N)}_v)<1$, where $P^{(N)}_v$ is defined in Equation \eqref{eq:PN}.
  \end{enumerate}
\end{prop}

\begin{proof}
 Let $x_0,x_1,\ldots,x_{n-1}$ be in $N$. We denote by $({Z_n})_n$ the embedded chain of $X$. On the set
$H=\BRA{Z_0=x_0,\ldots,Z_{n-1}=x_{n-1},Z_n\in M}$,
$$
I_{x_0}=\int_0^T\!\sigma(X^{x_0}_s)\,dB_s%
=\sum_{j=0}^{n-1}\sigma(x_j)\sqrt{\tau_{x_j}}G_j,
$$
where the random variables ${(G_j)}_j$, ${(\tau_{x_j})}_{j}$ are independent and $\cL(G_j)=\cN(0,1)$ and $\cL(\tau(x_j))=\cE(a(x_j))$. As a consequence, 
$$
\dE\PAR{e^{vI_{x_0}}\vert H}%
=\prod_{j=0}^{n-1}\dE\SBRA{\exp\PAR{\frac{v^2\sigma(x_j)^2}{2}\tau_{x_j}}}
=\prod_{j=0}^{n-1}\frac{1}{1-\beta(x_j)v^2}.
$$
 One just computes
  \begin{align*} 
  \dE(e^{v I_{x_0}})&=\sum_{\substack{n\geq
1\\x_1,\ldots,x_{n-1}\in N}}\dE(e^{vI_{x_0}}\,|\,
Z_1=x_1,\ldots,Z_{n-1}=x_{n-1},Z_n\in M)\times\\ 
&\phantom{\sum_{\substack{n\geq 1\\x_0,\ldots,x_{n-1}\in N}}
\dE(e^{v I}\,|\, Z_0)} \times
\dP_{x_0}(Z_1=x_1,\ldots,Z_{n-1}=x_{n-1},Z_n\in M))\\
&=\sum_{\substack{n\geq
1\\x_1,\ldots,x_{n-1}\in N}}
\frac{P(x_0,x_1)}{1-\beta(x_0)v^2}\cdots\frac{P(x_{n-2},x_{n-1})}{1-\beta(x_{n-2})v^2}\frac{P(x_{n-1},M)}{1-\beta(x_{n-1})v^2}\\ 
&=\sum_{\substack{n\ge
1\\\phantom{i_0,\ldots,i_{n-1}\in N}}}\delta_{x_0} {(P^{(N)}_v)}^{n-1}
\varphi\,,
\end{align*} 
for $\varphi(x)=\frac{1}{1-\beta(x)v^2}P(x,M)$. Notice that $\varphi$ is well-defined
since $v^2<1/\overline \beta$. Moreover it is positive because $X$ is irreducible recurrent,
so, for any $x_0\in N$ there exists a path that leads to $M$.

If $\rho(P^{(N)}_v)<1$,
then 
$$
\limsup_{n\rightarrow+\infty}{\big|\delta_{x_0}{(P^{(N)}_v)}^{n-1}\varphi\big|}^{1/n}%
\leq \limsup_{n\rightarrow+\infty}{\big\|{(P^{(N)}_v)}^n\big\|}^{1/n}<1\,,
$$
hence the series is convergent.

If $\rho_v:=\rho(P^{(N)}_v)\geq 1$, by Perron-Frobenius
theorem, there exists a probability measure $\nu_0$ with some positive
coefficients such that $\nu_0P^{(N)}_v=\rho_v \nu_0$, which
implies that
$$
\dE_{\nu_0}(e^{v I_\cdot })=\nu_0(\varphi) \sum_{n\geq 0}\rho_v^{n-1}=+\infty,
$$
since $\varphi$ is positive.
\end{proof}

\begin{rem}
 \label{rem:irrN} When $X$ is irreducible in restriction to $N$ (\textit{i.e.} the matrices $P_{v}^{(N)}$ are irreducible for any $v$), then $\dE(e^{v I_x})=+\infty$ for all $x\in N$ as soon as $\rho(P^{(N)}_{v})\geq1$. If this it not the case, the previous proposition just ensures that when $\rho(P^{(N)}_{v})\geq1$, then $\dE(e^{v I_x})=+\infty$ for some $x\in N$. Moreover, for any  $x,x'\in N$ such that $P(x,x')$ is positive then  $\dE(e^{v I_{x'}})=+\infty$ implies $\dE(e^{v I_x})=+\infty$.
\end{rem}

We now introduce the sub-process made of the positions of $(X,Y)$ at the 
successive hitting times of $M$. 

\begin{prop}\label{prop:souschaine}
For any $n\geq0$, let us define 
$$
U_n=X_{T_{2n}} 
\quad\text{and}\quad
V_n=Y_{T_{2n}}. 
$$
The process $(U,V)$ is a Markov chain on $F\times \dR$. More precisely, 
$$
V_{n+1}=M_n(U_n)V_n+Q_n(U_n),
$$
where the sequence of random vectors $\big({(M_{n}(x),Q_{n}(x))}_{x\in F}\big)$ is i.i.d., and independent of $(U_{n})$, with law given by 
\begin{align*}
M_n(x)&\overset{\cL}{=}\exp\PAR{-\int_{0}^{T_1}\!\lambda(X_r^x)\,dr}\\
Q_n(x)&\overset{\cL}{=}\int_{0}^{T_{1}}\!\sigma(X_s^x)\exp\PAR{-\int_s^{T_{1}}\!\lambda(X_r^x)\,dr}\,dB_s
+ \int_{T_{1}}^{T_{2}}\!\sigma(X_s^x)\,dB_s.
\end{align*}
 For any $v<v_c$ where $v_c=\sup\BRA{v,\ \rho(P_v^{(N)})<1}$, we have 
 $$
 \sup_{n\geq 0}\dE\PAR{e^{v\ABS{V_n}}}<+\infty.
 $$
 Moreover, if $v\geq v_c$, this supremum is infinite.
\end{prop}

\begin{proof}
 The fact that $(U,V)$ is a recurrent Markov chain is a straightforward application of the Markov property for $X$. 
 
 Let us introduce $\overline M_n=\max_{x\in F}M_n(x)$ and $\overline Q_n=\max_{x\in F}\ABS{Q_n(x)}$. The random variables ${((\overline M_n,\overline Q_n))}_{n\geq 0}$ are i.i.d. Define the sequence ${(\overline V_n)}_{n\geq 0}$ by
$$
\overline V_0=\ABS{V_0}%
\quad\text{and} \quad%
 \overline{V}_{n+1}=\overline M_n\overline{V}_n+\overline Q_n
 \quad\text{for } n\geq 1.
 $$
 The domain of the Laplace transforms of ${(\overline V_n)}_{n\geq 0}$ 
 is known thanks to the exhaustive study \cite{iksanov}. 
 Since  $\dP(\overline Q_n=0)<1$, $\dP(0<\overline M_n<1)=1$ 
 and for any $c\in\dR$, $\dP(\overline Q_n+\overline M_n c=c)<1$, 
 \cite[Theorem 1.6]{iksanov} ensures in particular that  ${(\dE\exp\PAR{v\overline V_n})}_n$ 
 is uniformly bounded as soon as the Laplace transform $L_{\overline Q}$ of $\overline Q$ is finite. At last, for any $v\geq0$,
$$
\sup_{x\in F}\dE\PAR{e^{v\ABS{Q(x)}}}\leq%
\dE\PAR{e^{v\overline Q}}=\dE\PAR{\sup_{x\in F}e^{v\ABS{Q(x)}}}%
\leq \sum_{x\in F}\dE\PAR{e^{v\ABS{Q(x)}}}.
$$
Thus $L_{\overline Q}(v)$ is finite if and only if $\dE\PAR{e^{v\ABS{Q(x)}}}$ is finite for any $x\in F$. 
Since $\ABS{V_n}\leq \overline{V}_n$ for all $n\geq 0$, then  
$$
\sup_{n\geq 0}\dE\PAR{e^{v\ABS{V_n}}}<+\infty
$$
as soon as $L_{\overline Q}(v)$ is finite.


On the other hand, choose $v$ such that there exists $x_0\in F$ such that 
$\dE\PAR{e^{v\ABS{Q(x_0)}}}$ is infinite. Then, for any $n\geq 0$,
\begin{align*}
\dE\PAR{e^{v\ABS{V_{n+1}}}}&\geq \dE\PAR{e^{v\ABS{V_{n+1}}}\ind_\BRA{U_n=x_0}}\\
&\geq \dE\PAR{e^{-v\ABS{V_n}}e^{v\ABS{Q_n(x_0)}}\ind_\BRA{U_n=x_0}}\\
&\geq \dE\PAR{\ind_\BRA{U_n=x_0}e^{-v\ABS{V_n}}}\dE\PAR{e^{v\ABS{Q_n(x_0)}}}.
 \end{align*}
The recurrence of $U$ ensures that $\BRA{n\geq 0,\ \dE\PAR{e^{v\ABS{V_n}}}=+\infty}$ is infinite. 

The last point is to show that $L_{\overline Q}(v)$ is finite if and only if $v<v_c$ where 
$v_c$ is defined by \eqref{eq:vc}. For any $x\in F$, 
the random variable $Q_n(x)$ is symmetric and its Laplace transform is 
finite as soon as, for any $\tilde x\in N$, the Laplace transform of 
 $$
I_{\tilde x}=  \int_{0}^{T}\!\sigma(X_s^{\tilde x})\,dB_s
  $$
is finite, which is true for $\ABS{v}<v_c$. Indeed, we have for any $v$ 
\begin{equation}
\dE\PAR{e^{vQ_n(x)}\vert\cF_{T_1}}=\exp\PAR{v\int_{0}^{T_{1}}\!\sigma(X_s^x)\exp\PAR{-\int_s^{T_{1}}\!\lambda(X_r^x)\,dr}\,dB_s}\dE\PAR{e^{vI_{\tilde x}}}_{\vert \tilde x=X_{T_1}}.\label{eq:expQJ}
\end{equation}
Proposition \ref{prop:spectral} ensures that, if $\ABS{v}<v_c$ then 
$$
 \dE\PAR{e^{vQ_n(x)}}\leq C(v)\dE\PAR{\exp\PAR{\frac{v^2}{2}\int_{0}^{T_{1}}\!\sigma(X_s^x)^2%
 \exp\PAR{-2\int_s^{T_{1}}\!\lambda(X_r^x)\,dr}\,ds }}.
$$
Denoting $\overline \sigma_M=\max_{x\in M}\sigma(x)$ and $\underline \lambda_M=
\min_{x\in M}\lambda(x)$, one has 
$$ 
\dE\PAR{e^{vQ_n(x)}}%
\leq C(v)\exp\PAR{\frac{\overline \sigma_M^2}{4\underline \lambda_M}v^2}. 
$$
By the way, $L_{\overline Q}$ is finite on $(-\infty, v_c)$.

We assume now that $v\geq v_{c}$. From Proposition \ref{prop:spectral}, we know that, in this case, the set $G=\{x\in N, \ \dE(e^{vI_{x}})=+\infty\}$ is non empty. Using the irreducibility of $X$ and Remark \ref{rem:irrN}, one notices that there exists $x_{0}\in F$ such that $\dP(X_{T_{1}}^{x_{0}}\in G)>0$. From this remark and \eqref{eq:expQJ}, one has $\dE(e^{vQ_{n}(x_{0})})=+\infty$ which conclude the proof.
\end{proof}

Let us now extend this result to the whole process $Y$. 

\begin{thm}\label{th:TLY}
For any $v<v_c$ where $v_c=\sup\BRA{v,\ \rho(P_v^{(N)})<1}$, we have 
 $$
 \sup_{t\geq 0}\dE\PAR{e^{v\ABS{Y_t}}}<+\infty.
 $$
 Moreover, if $v\geq v_c$, then this supremum is infinite.
\end{thm}

\begin{proof}
 Choose $t>0$. We have  
$$
 \dE\PAR{e^{v\ABS{Y_t}}}=\sum_{n=0}^\infty\dE\PAR{e^{v\ABS{Y_t}}\ind_\BRA{T_{2n}\leq t<T_{2n+2}}}.
$$
We write, for $0\leq v<v_c$, 
 \begin{align*}
 \dE\PAR{e^{v\ABS{Y_t}}\ind_\BRA{T_{2n}\leq t<T_{2n+2}}}&%
= \dE\PAR{\dE\PAR{e^{v\ABS{Y_t}}\ind_\BRA{T_{2n}\leq t<T_{2n+2}}\vert\cF_{T_{2n}}\vee \cF_t^X}}
\end{align*}
 As in the proof of Proposition \ref{prop:souschaine}, 
\begin{align*}
\dE\PAR{e^{v\ABS{Y_t}}\ind_\BRA{T_{2n}\leq t<T_{2n+2}}\vert\cF_{T_{2n}}\vee \cF_t^X}
&\leq C(v)\exp\PAR{\frac{\overline \sigma_M^2}{4\underline \lambda_M}v^2}e^{v\ABS{Y_{T_{2n}}}}
\dE\PAR{\ind_\BRA{T_{2n}\leq t<T_{2n+2}}\vert\cF_{T_{2n}}\vee \cF_t^X}.
\end{align*}
By the Markov property applied to $X$, 
$$
 \dE\PAR{e^{v\ABS{Y_t}}\ind_\BRA{T_{2n}\leq t<T_{2n+2}}}\leq%
C(v)\exp\PAR{\frac{\overline \sigma_M^2}{4\underline \lambda_M}v^2}
\dE\PAR{e^{v\ABS{Y_{T_{2n}}}}}
\dP\PAR{T_{2n}\leq t<T_{2n+2}}.
$$

Then, for $0\leq v<v_c$, 
$$
\dE\PAR{e^{v\ABS{Y_t}}}\leq 
C(v)
\exp\PAR{\frac{\overline \sigma_M^2}{4\underline \lambda_M}v^2}
\sup_{n\geq 0}\dE\PAR{e^{v\ABS{Y_{T_{2n}}}}}.
$$
The generalisation of the case $v\geq v_{c}$ to the whole process is immediate.
\end{proof}


\section{Polynomial moments for the switched diffusion}\label{sec:yaothereturn}

We denote by $A_p$ the matrix $A-p\Lambda$ where $\Lambda$ is the diagonal matrix with diagonal
$(\lambda(1),\ldots,\lambda(d))$ and associate to $A_p$ the quantity
$$
\eta_p:=-\max_{\gamma\in\mathrm{Spec}(A_p)} \mathrm{Re\ }\gamma.
$$
The main goal of this section is to establish the equivalence between the positivity of $\eta_p$ and the existence of a $p^\text{th}$ moment for the invariant measure $\nu$ of $Y$. We will also give the proof of Point 1 of Theorem \ref{th:tricho}.

Using classical ideas from spectral theory, we first relate $\eta_{p}$ with exponential functionals of $\lambda$ along the trajectories of $X$:

\begin{prop}\label{prop:encadrement} For any $p>0$, there exist $0<C_{1}(p) < C_{2}(p) <+\infty$ such that, for any initial probability measure $\pi$ on $E$, any $t>0$,
\begin{equation}
\label{eq:conv-expo}
C_{1}(p)e^{-\eta_{p}t} \leq \dE_\pi\PAR{\exp{\PAR{-\int_0^t\!p\lambda(X_u)\,du}}}\leq C_{2}(p)e^{-\eta_{p}t}.
\end{equation}
\end{prop}

\begin{proof}
 Let us define, for any $p>0$ and $t>0$,
the matrix $A_{(p,t)}$ by
$$
A_{(p,t)}(x,\tilde x)=%
\dE_x\PAR{\exp\PAR{-\int_0^t\!p\lambda(X_u)\,du}%
 \ind_\BRA{X_t=\tilde x}}.
$$
On the one hand, one remarks that 
\begin{equation}
\label{eq:feynman}
\dE_\pi\PAR{\exp{\PAR{-\int_0^t\!p\lambda(X_u)\,du}}}=\pi A_{(p,t)} \textbf{1}
\end{equation}
where the coordinates of $\textbf{1}$ are all equal to 1 and $\pi$ is a probability measure on $E$ seen as a row vector.

On the other hand, a simple application of the Feynman-Kac formula shows that 
$A_{(p,t)}=e^{tA_{p}}$. This fact relates the spectra of $A_{p}$ and $A_{(p,t)}$. 
In particular, $\rho(A_{(p,t)})=e^{-\eta_{p}t}$ and, since all coefficients of $A_{(p,t)}$ are positive, we can apply the Perron-Frobenius Theorem to ensure that $-\eta_{p}$ is a simple eigenvalue of $A_{p}$, all other eigenvalues having a strictly smaller real part. 
Let $\xi_{p}<-\eta_{p}$ be an upper bound for the real parts of these other eigenvalues.

We then define $\pi_{p}$ (resp. $\varphi_{p}$) the left (resp. right) eigenvector associated to $-\eta_{p}$, with positive coefficients, normalized such that $\pi_{p}(\mathbf{1})=1$ (resp. $\pi_{p}(\varphi_{p})=1$). Applying \cite[Thm VII.1.8]{dunsch}, we get that for any $t>0$
$$
e^{tA_{p}}=e^{-\eta_{p}t}\varphi_{p}\pi_{p}+R_{p}(t),
$$
with $\|R_{p}(t)\|_{\infty}\leq P_{p}(t)e^{\xi_{p}t}$, $P_{p}(t)$ being a polynomial of degree less than $d$. This gives 
$$
\pi e^{tA_{p}}\mathbf{1}=e^{-t\eta_{p}}(\pi(\varphi_{p})+e^{t\eta_{p}}\pi R_{p}(t)\mathbf{1})
$$
hence
$$
e^{-t\eta_{p}}(\pi(\varphi_{p})-P_{p}(t)e^{t(\eta_{p}+{\xi_{p}})})
\leq \pi e^{tA_{p}}\mathbf{1}
\leq e^{-t\eta_{p}}(\pi(\varphi_{p})+P_{p}(t)e^{t(\eta_{p}+{\xi_{p}})}).
$$
This estimate gives \eqref{eq:conv-expo} thanks to \eqref{eq:feynman} and to the fact that $P_{p}(t)e^{t(\eta_{p}+{\xi_{p}})}$ tends to $0$ as $t$ tends to infinity.
\end{proof}

Let us now study the function $p\mapsto \eta_p$. 
\begin{prop}\label{prop:eta}$\ $
\begin{enumerate}
\item The function $p\mapsto \eta_{p}$ is smooth and concave on $\dR_{+}$. Its derivative at $p=0$ is equal to 
$$
\sum_{x\in E}\lambda(x)\mu(x)>0,
$$
and $\eta_p/p$ tends to $\underline \lambda$ as $p$ goes to infinity.
\item We have the following dichotomy:
 \begin{itemize}
 \item if $\underline \lambda\geq 0$, then for all $p>0$, $\eta_p>0$,
 \item if $\underline\lambda< 0$, there is $\kappa\in(0,\min\{-a(x)/\lambda(x),\lambda(x)<0\})$ such that $\eta_p>0$ for $p<\kappa$ and $\eta_p<0$ for $p>\kappa$.
 \end{itemize}
\end{enumerate}
\end{prop}

\begin{proof}
The smoothness of the functions $\eta_{p}$, $\pi_{p}$ and $\varphi_{p}$ are classical results of perturbation theory (see for example \cite[chapter 2]{kato}). Since $\pi_{p}A_{p}=-\eta_{p}\pi_{p}$, $\pi_{p}\mathbf{1}=1$ and $A\mathbf{1}=0$, one has
\begin{equation}
\label{eq:expeta}
\eta_{p}=-\pi_{p}A_{p}\mathbf{1}=p\pi_{p}\Lambda\mathbf{1}=p\sum_{x\in E}\pi_{p}(x)\lambda(x).
\end{equation}
Differentiating this relation gives
$
\eta'_{p}=\pi_{p}\Lambda\mathbf{1}+p\pi'_{p}\Lambda\mathbf{1}
$.
In particular, $\eta'_{0}=\mu\Lambda\mathbf{1}=\sum_{x\in E}\mu(x)\lambda(x)$, since $\pi_{0}=\mu$.

\medskip

We turn to the proof of the concavity of $\eta_{p}$. We only have to remark that, 
for any $t>0$ and any $x\in E$, 
$$
p\mapsto M^{(x)}_{t}(p)=\frac1t\log\dE_{x}\PAR{\exp{\PAR{-p\int_0^t\!\lambda(X_u)\,du}}}
$$
 is a convex function, as a log-Laplace transform (for example using H\"older's inequality). But \eqref{eq:conv-expo} implies that $M_{t}^{(x)}$ converges to $-\eta_{p}$, hence $\eta_{p}$ is concave as a limit of concave functions.

Obviously, one has, for any $t>0$ and $p>0$, 
$M_t^{(x)}(p)\leq -p\underline \lambda$
and $\eta_p$ is greater than $p\underline \lambda$. 
On the other hand, denoting by $T$ the first jump time of $(X_{t})$, one has
\begin{align*}
 M^{(x)}_{t}(p)
 \geq & \frac1t\log\dE_{x}\PAR{\exp{\PAR{-p\int_0^t\!\lambda(X_u)\,du}}\ind_\BRA{T>t}}\\
 \geq &-p\lambda(x)+\frac1t\log\dP_{x}(T>t)=-p\lambda(x)-a(x).
\end{align*}
When $t$ goes to infinity, one gets for any $p>0$
\begin{equation}\label{eq:majeta}
\eta_{p}\leq \min_{x\in E}(a(x)+p\lambda(x)).
\end{equation}
In particular, $\eta_p/p$ goes to $\underline\lambda$ as $p$ goes to infinity. 
\medskip

The fact that, when $\underline{\lambda}\geq0$, $\eta_{p}$ is always positive is clear from \eqref{eq:expeta}. 

\medskip

When $\underline{\lambda}<0$, for $p$ small enough, $\eta_{p}>0$ since its derivative at $p=0$ is positive.  But in this case, we can check that $\eta_{p}<0$ for $p$ large enough. Equation \eqref{eq:majeta}
 implies that $\eta_{p}<0$ as soon as $p>\min_{x\in E,\lambda(x)<0}-a(x)/\lambda(x)$. This provides the upper bound for $\kappa$. 

With the concavity of $\eta_{p}$, these considerations are sufficient to ensure that $\eta_{p}$ as a unique zero $\kappa$, being positive before and negative after.
\end{proof}

\begin{rem}\label{rem:kappa}
 The relation $\eta_{\kappa}=0$ implies that $(A-\kappa\Lambda)\varphi_{\kappa}=0$ which can be rewritten as $M_{\kappa}\varphi_{\kappa}=\varphi_{\kappa}$ ($M_{\kappa}$ being the matrix defined in \eqref{eq:Ms}). This ensures that $\rho(M_{\kappa})=1$ since $M_{\kappa}$ is non-negative irreducible and $\varphi_{\kappa}$ is positive. By the way our characterization of $\kappa$ in Theorem \ref{th:tricho} is equivalent to the one given by de~Saporta and Yao in Point 1. of Theorem \ref{th:kappa}. 
\end{rem}

It is known from \cite{guyon,saporta-yao} that the invariant measure $\nu$ of $Y$ has $p^\text{th}$ finite moment if and only if $p<\kappa$. Their proof is based on a time discretization of the process $(X,Y)$  together with generic results on the ergodicity of discrete time Markov processes and renewal theory (see \cite{saporta}). The previous propositions provide a direct and simple characterization of the critical moment of $\nu$. 

\begin{prop}\label{prop:moments}
 For any $p>0$ such that $\eta_p>0$ (\emph{i.e.} $p<\kappa$), and any initial measure such that the second marginal has a $p^\text{th}$ finite moment, one has
 $$
 \sup_{t\geq 0}\dE\PAR{\ABS{Y_t}^p}<+\infty
\quad\text{and}\quad
\int\ABS{y}^p\,\nu(dy)<+\infty.
$$ 
On the other hand, for any $p$ such that $\eta_p\leq 0$ (\emph{i.e.} $p\geq\kappa$) and any initial condition,  
 $$
 \lim_{t\rightarrow\infty}\dE\PAR{\ABS{Y_t}^p}=+\infty
\quad\text{and}\quad
\int\ABS{y}^p\,\nu(dy)=+\infty.
$$
\end{prop}

\begin{proof}
 Let us assume that $p\geq 2$. If it is not the case, one has to replace the function $y\mapsto \ABS{y}^p$ by the $\cC^2$ function $y\mapsto \frac{\ABS{y}^{p+2}}{1+\ABS{y}^2}$. Choose $T>0$. It\^o's formula ensures that
 \begin{equation}
   \label{eq:itos}
   d\ABS{Y_t}^p%
   =\PAR{-p\lambda(X_t)\ABS{Y_t}^p+%
     \frac{p(p-1)}{2}\sigma(X_t)^2\ABS{Y_t}^{p-2}}\,dt%
   +p\sigma(X_t)Y_t\ABS{Y_t}^{p-2}\,dB_t.
 \end{equation}
 Let us denote by $\alpha_p$ the function defined on $[0,T]$ by
 $$
 \alpha_p(t)=\dE\PAR{\ABS{Y_t}^p|\cF_T^X}.
 $$
 Taking the expectation of
 \eqref{eq:itos} conditionnally to $X$ leads to
$$
\alpha_p'(t)=-p\lambda(X_t)\alpha_p(t)+%
\frac{p(p-1)}{2}\sigma^2(X_t)\alpha_{p-2}(t),
$$
 since $B$ and $X$ are independent. 
For any $\varepsilon>0$, there exists $c$ 
such that 
$$
\alpha_p'(t)\leq%
(-p\lambda(X_t)+\varepsilon)\alpha_p(t)+c.
$$
This implies that
$$
\alpha_p(t)\leq %
\alpha_p(0)e^{\int_0^t\!(-p\lambda(X_r)+\varepsilon)\,dr}%
+c\int_0^t\!  e^{\int_u^t\!(-p\lambda(X_r)+\varepsilon)\,dr}\,du.
$$
One has to take the expectation and use \eqref{eq:conv-expo} to get
for any $p>2$ such that $\eta_p>0$
$$
\dE\PAR{\ABS{Y_t}^p}\leq%
C_2(p)\dE\PAR{\ABS{Y_0}^p}e^{(-\eta_p+\varepsilon)t}%
+c\ C_2(p)\int_0^t\!   e^{-(-\eta_p+\varepsilon)u}\,du.
$$
If $\varepsilon<\eta_p$ then $\sup_{t>0}\dE(\ABS{Y_t}^p)$ is finite.

If $p=\kappa$, one has 
$$
\alpha_\kappa'(t)=-\kappa\lambda(X_t)\alpha_\kappa(t)+%
\frac{\kappa(\kappa-1)}{2}\sigma^2(X_t)\alpha_{\kappa-2}(t).
$$
Then 
\begin{align*}
\alpha_\kappa(t)&=\int_0^t\!e^{-\kappa\int_s^t\!\lambda(X_u)\,du}\kappa(\kappa-1)\sigma(X_s)^2\alpha_{\kappa-2}(s)\,ds
+\dE\PAR{\ABS{Y_0}^\kappa}e^{-\kappa\int_0^t\!\lambda(X_u)\,du}\\
&\geq \kappa(\kappa-1)\underline \sigma^2\int_0^t\!e^{-\kappa\int_s^t\!\lambda(X_u)\,du}\alpha_{\kappa-2}(s)\,ds.
\end{align*}
As a consequence, using Proposition \ref{prop:encadrement} and the relation $\eta_\kappa=0$ (see Proposition \ref{prop:eta}), 
\begin{align*}
\dE\PAR{\ABS{Y_t}^\kappa}&\geq \kappa(\kappa-1)\underline \sigma^2\int_0^t\!
\dE\PAR{\alpha_{\kappa-2}(s)\dE\PAR{e^{-\kappa\int_s^t\!\lambda(X_u)\,du}\vert\cF^X_s}}\,ds\\
&\geq \kappa(\kappa-1)\underline \sigma^2C_1(\kappa)\int_0^t\!
\dE\PAR{\ABS{Y_s}^{\kappa-2}}\,ds.
\end{align*}
From the first part of the proof, 
$$
\lim_{s\rightarrow\infty}\dE\PAR{\ABS{Y_s}^{\kappa-2}}=\int\! \ABS{y}^{\kappa-2}\,\nu(dy)>0.
$$
By the way, 
$$
\lim_{t\rightarrow\infty}\dE\PAR{\ABS{Y_t}^\kappa}=+\infty,
$$
and the $\kappa^{th}$ moment of $\nu$ is infinite. This is also true for the $p^{th}$ moment for any $p>\kappa$. 
\end{proof}

\section{Convergence to equilibrium for the switched diffusion} \label{sec:moment-convergence}

Under the assumption that $\nu$ has a finite $p^{th}$ moment, one can
establish an exponential convergence of $(X,Y)$ to its invariant measure 
in terms of mixed total variation (for $X$) and $W_p$ Wasserstein distance (for $Y$).

Let us start with the easiest case, assuming that $\cL(X_0)=\cL(\tilde X_0)$. 

\begin{proof}[Proof of Theorem \ref{th:Xergo}]
  Let $y$ and $\tilde y$ be two real numbers. We couple two
  trajectories of $(X,Y)$ starting at $(x,y)$ and $(x,\tilde y)$ by
  choosing the same first components and the same Brownian motion to
  drive $Y$ and $\tilde Y$. In other words, we compare
  $(X_t,Y_t)^{x,y}$ and $(\tilde X_t,\tilde Y_t)^{x,\tilde y}$ where 
  $$
  \begin{cases}
    X_t=\tilde X_t,&\\
    \ds{Y_t=y-\int_0^t\!\lambda(X_u)Y_u\,du+\int_0^t\sigma(X_u)\,dB_u}& \\
    \ds{\tilde Y_t=\tilde y-\int_0^t\!\lambda(X_u)\tilde
      Y_u\,du+\int_0^t\sigma(X_u)\,dB_u.}&
  \end{cases}
  $$
  Then,
  $$
  d\PAR{Y_t-\tilde Y_t}=-\lambda(X_t)(Y_t-\tilde Y_t)\,dt
  $$
  and
  $$
  \ABS{Y_t-\tilde Y_t}^p=\ABS{y-\tilde y}^p - \int_0^t\!
  p\lambda(X_u)\ABS{Y_u-\tilde Y_u}^p\,du.
  $$
  As a conclusion, \eqref{eq:conv-expo} ensures that 
  $$
  \dE_{(x,y),(x,\tilde y)}\PAR{\ABS{Y_t-\tilde Y_t}^p}=%
  \dE_x\PAR{\exp{\PAR{-\int_0^t\!p\lambda(X_u)\,du}}}\ABS{y-\tilde y}^p%
  \leq C_2(p)e^{-\eta_p t}\ABS{y-\tilde y}^p.
  $$
  Then, for any coupling $\Pi$ of $\cL(Y_0)$ and $\cL(\tilde Y_0)$,
  $$
  W_p\PAR{\cL(Y_t),\cL(\tilde Y_t)}^p\leq%
  C_2(p)e^{-\eta_p t} \int\!\ABS{y-\tilde y}^p\,\Pi(d(y,\tilde y)).
  $$
  Taking the infimum over $\Pi$ provides the result.
\end{proof}

Let us turn to the general case.

\begin{thm}\label{th:expo-general}
Consider two processes $(X,Y)$ and $(\tilde X,\tilde Y)$ with respective initial laws $\pi$ and $\tilde \pi$ two probability measures on $E\times \dR$ such that the second marginal has a finite $\theta^{th}$ moment with $\theta<\kappa$ (with $\kappa=+\infty$ if $\underline\lambda\geq 0$).  For any $p<\theta$,  we have 
 $$
W_p\PAR{\cL(Y_t),\cL(\tilde Y_t)}^p\leq %
C_2(p)(1-p_c)^{1-p/\theta}M_0(\theta)^{p/\theta}
\exp\PAR{-\frac{\gamma\eta_p}{(1-p/\theta)\gamma+\eta_p}t}%
+p_c\overline W_p^p e^{-\eta_p t}, 
$$
where  
\begin{align*}
p_c&=\sum_{x\in E}\mu_0(x)\wedge\tilde\mu_0(x)= 1-d_{\mathrm{TV}}\PAR{\cL(X_0),\cL(\tilde X_0)},\\
M_0(\theta)^{p/\theta}&= 2^p\PAR{\sup_{t\geq 0}\dE\PAR{\ABS{Y_t}^\theta}+\sup_{t\geq 0}\dE\PAR{\vert\tilde Y_t\vert^\theta} }^{p/\theta},\\
\overline W_p&=\max_{x\in E}  W_p\PAR{\cL(Y_0\vert X_0=x),\cL(\tilde Y_0\vert \tilde X_0=x)},
\end{align*}
and $\gamma$ is such that 
$$
d_{\mathrm{TV}}(\cL(X_t),\cL(\tilde X_t))\leq e^{-\gamma t}d_{\mathrm{TV}}\PAR{\cL(X_0),\cL(\tilde X_0)}.
$$
\end{thm}

\begin{rem}
 This estimate can be improved and simplified if $\underline \lambda> 0$.  
 In this case, one can write instead of \eqref{eq:decomp} that 
 $$
   \dE_{(x,y),(\tilde x,\tilde y)}\PAR{\ABS{Y_t-\tilde Y_t}^p%
  \ind_\BRA{T\geq \alpha t}} 
\leq C\dP(T\geq \alpha t)
 $$
 thanks to the explicit expression \eqref{eq:explicit} of $Y$. Since $p\underline \lambda\leq \eta_p$ this leads to
  $$
W_p\PAR{\cL(Y_t),\cL(\tilde Y_t)}^p\leq %
C(p) (1-p_c)\exp\PAR{-\frac{\gamma p\underline \lambda}{\gamma+p\underline \lambda}t}%
+p_c\overline W_p^pe^{-p\underline \lambda t}.
$$
\end{rem}

\begin{proof}[Proof of Theorem \ref{th:expo-general}]
We have to consider the case $X_0\neq\tilde X_0$.  Given $x,\tilde x\in E$ (with $x\neq\tilde x$) and $y,\tilde y\in\dR$, we introduce the three independent processes $({X_t})_{t\geq 0}$, ${(\overline X_t)}_{t\geq 0}$ and ${(B_t)}_{t\geq 0}$ where the first one is a chain starting at $x$, the second one is a chain starting at $\tilde x$ and the last one is a standard Brownian motion. The process $\tilde X$ is defined as follows: 
$$
\tilde X_t=
\begin{cases}
 \overline X_t &\text{if }t\leq T,\\
 X_t               &\text{if }t>T,
\end{cases}
$$
where $T=\inf\BRA{t>0,\ X_t=\overline X_t}$. It is well known (since $X$ is a finite irreducible continuous time Markov chain) that 
 there exists $\gamma>0$ such that 
 $$
 \sup_{x,\tilde x\in E}\dP_{x,\tilde x}(T>t)\leq e^{-\gamma t}. 
 $$
Let us now define for any $t\geq 0$, 
\begin{align*}
 Y_t&=y e^{-\int_0^t\!\lambda(X_u)\,du}+%
\int_0^t e^{-\int_u^t\!\lambda(X_v)\,dv}\sigma(X_u)\,dB_u,\\
\tilde Y_t&=\tilde ye^{-\int_0^t\!\lambda(\tilde X_u)\,du}+%
\int_0^te^{-\int_u^t\!\lambda(\tilde X_v)\,dv}\sigma(\tilde X_u)\,dB_u.
\end{align*}

Let us denote, for any $p<\kappa$ and $y,\tilde y\in\dR$, 
$$
C(p,x,y)=\sup_{t\geq 0}\dE_{x,y}\PAR{\ABS{Y_t}^p}%
\quad\text{and}\quad
C(p,x,y,\tilde x,\tilde y)=2^p\PAR{C(p,x,y)+C(p,\tilde x,\tilde y)}.
$$
Let $\alpha\in (0,1)$ and $s$ be the conjugate of $\theta/p$.  Theorem \ref{th:Xergo} ensures that
\begin{align}
  \dE_{(x,y),(\tilde x,\tilde y)}\PAR{\ABS{Y_t-\tilde Y_t}^p}&=%
  \dE_{(x,y),(\tilde x,\tilde y)}\PAR{\ABS{Y_t-\tilde Y_t}^p%
  \PAR{\ind_\BRA{T\geq \alpha t}+\ind_\BRA{T< \alpha t}}}\nonumber\\%
&\leq C(\theta,x,y,\tilde x,\tilde y)^{p/\theta}e^{-\gamma\alpha t/s}\label{eq:decomp}\\%
 &\quad\quad+ \dE_{(x,y),(\tilde x,\tilde y)}\PAR{\ABS{Y_T-\tilde Y_T}^p%
  C_2(p)e^{-\eta_p(t-T)}\ind_\BRA{T< \alpha t}}\nonumber\\%
&\leq C_2(p) C(\theta,x,y,\tilde x,\tilde y)^{p/\theta}\PAR{e^{-\gamma\alpha t/s}+  %
 e^{-\eta_p(1-\alpha)t}}\nonumber. 
\end{align}
Optimizing over $\alpha$ in order to have $\gamma\alpha/s=\eta_p(1-\alpha)$ \emph{i.e.}
$\alpha=\frac{s\eta_p}{\gamma+s\eta_p}$ leads to 
$$
  \dE_{(x,y),(\tilde x,\tilde y)}\PAR{\ABS{Y_t-\tilde Y_t}^p}%
\leq C_2(p)C(\theta,x,y,\tilde x,\tilde y)^{p/\theta}\exp\PAR{-\frac{\gamma\eta_p}{\gamma+s\eta_p}t}.
$$

Let us now turn to the case of general initial conditions. Let $\pi_0$ and $\tilde \pi_0$ be two probability measures on $E\times \dR$ such that the second marginal has a finite $\theta^{th}$ moment. Let us start coupling the marginals $\mu_0$ and $\tilde \mu_0$ on $E$. Define the coupling probability $p_c$
$$
p_c=\sum_{x\in E}\mu_0(x)\wedge\tilde\mu_0(x),
$$
and $D=\BRA{x\in E,\ \mu_0(x)\geq\tilde\mu_0(x)}$. We introduce the random variables $U$, $V$, $W$ and $Z$ such that for any $x\in E$
\begin{align*}
\dP(U=x)&=\frac{\mu_0(x)\wedge\tilde\mu_0(x)}{p_c},\\%
\dP(V=x)&=\frac{\mu_0(x)-\tilde\mu_0(x)}{1-p_c}\ind_{D}(x),\\%
\dP(W=x)&=\frac{\tilde\mu_0(x)-\mu_0(x)}{1-p_c}\ind_{D^c}(x),%
\end{align*}
and $\dP(Z=1)=1-\dP(Z=0)=p_c$, $Z$ being independent of $(U,V,W)$. We can now define  
$$
X_0=
\begin{cases}
U&\text{if }Z=1,\\
V&\text{if }Z=0, 
\end{cases}
\quad
\tilde X_0=
\begin{cases}
U&\text{if }Z=1,\\
W&\text{if }Z=0. 
\end{cases}
 $$
We check by a standard computation that the law of $X_0$ (resp. $\tilde X_0$) is $\mu_0$ (resp. $\tilde \mu_0$). 

Now, for any $x\in E$,  let us introduce two random variables $Y_0^x$ and $\tilde Y_0^x$, independent of $(U,V,W,Z)$ such that 
$$
\dE\PAR{\ABS{Y_0^x-\tilde Y_0^x}^\theta}=%
W_\theta\PAR{\cL(Y_0\vert X_0=x),\cL(\tilde Y_0\vert \tilde X_0=x)}^\theta.
$$
With this construction $(X_0,Y_0^{X_0})$ has law $\pi_0$ and $(\tilde X_0,\tilde Y_0^{\tilde X_0})$ has law $\tilde \pi_0$. We consider the processes $(X,Y)$ and $(\tilde X,\tilde Y)$ with these initial conditions, the sticky Markov chains and the same Brownian motion. Thanks to the previous computations, we have
\begin{align*}
  \dE\PAR{\ABS{Y_t-\tilde Y_t}^p}&=%
 \dE\PAR{\ABS{Y_t-\tilde Y_t}^p%
 \PAR{\ind_\BRA{X_0=\tilde X_0}+\ind_\BRA{X_0\neq\tilde X_0}}}\\%
 & \leq \dE\PAR{\ind_\BRA{X_0=\tilde X_0}\ABS{Y_0^{X_0}-\tilde Y_0^{\tilde X_0}}^p}e^{-\eta_p t}\\%
 &\quad+C_2(p)\dE\PAR{\ind_\BRA{X_0\neq\tilde X_0}%
 C(\theta,X_0,Y_0^{X_0},\tilde X_0,\tilde Y_0^{\tilde X_0})^{p/\theta}}%
 \exp\PAR{-\frac{\gamma\eta_p}{\gamma+s\eta_p}t}.
\end{align*}
On the one hand, we have 
\begin{align*}
 \dE\PAR{\ind_\BRA{X_0=\tilde X_0}\ABS{Y_0^{X_0}-\tilde Y_0^{\tilde X_0}}^p}&= 
 \dE\PAR{\ind_\BRA{X_0=\tilde X_0}\dE\PAR{\ABS{Y_0^{X_0}-\tilde Y_0^{X_0}}^p\vert X_0=\tilde X_0}}\\
 &\leq p_c\overline W_p^p,
\end{align*}
 where $\overline W_p=\max_{x\in E}  W_p\PAR{\cL(Y_0\vert X_0=x),\cL(\tilde Y_0\vert \tilde X_0=x)}$.
On the other hand, 
$$
\dE\PAR{\ind_\BRA{X_0\neq\tilde X_0}C(\theta,X_0,Y_0^{X_0},\tilde X_0,\tilde Y_0^{\tilde X_0})^{p/\theta}}\leq %
\dP(X_0\neq\tilde X_0)^{1/s}\dE\PAR{C(\theta,X_0,Y_0^{X_0},\tilde X_0,\tilde Y_0^{\tilde X_0})}^{p/\theta}.
$$
As a conclusion we get the following bound: 
$$
W_p\PAR{\cL(Y_t),\cL(\tilde Y_t)}^p\leq %
C_2(p)(1-p_c)^{1/s}M_0(\theta)^{p/\theta}\exp\PAR{-\frac{\gamma\eta_p}{\gamma+s\eta_p}t}%
+p_c^{1/s}\overline W_\theta^{p/\theta} e^{-\eta_p t}, 
$$
where 
$$
M_0(\theta)^{p/\theta}= 2^p\PAR{\dE\PAR{C(\theta,X_0,Y_0)}+\dE\PAR{C(\theta,\tilde X_0,\tilde Y_0)} }^{p/\theta}.
$$
\end{proof}

\addcontentsline{toc}{section}{\refname}%
{ \footnotesize
\bibliography{MarkovSwitching}
}
\bibliographystyle{amsplain}

\bigskip

\begin{flushright}\texttt{Compiled \today.}\end{flushright}


{\footnotesize %
 \noindent Jean-Baptiste \textsc{Bardet}
e-mail: \texttt{jean-baptiste.bardet(AT)univ-rouen.fr}
  
 \medskip

 \noindent\textsc{UMR 6085 CNRS Laboratoire de Math\'ematiques Rapha\"el Salem (LMRS)\\
Universit\'e de Rouen,
Avenue de l'Universit\'e, BP 12,
F-76801 Saint Etienne du Rouvray}

 \bigskip
  
 \noindent H\'el\`ene \textsc{Gu\'erin},
e-mail: \texttt{helene.guerin(AT)univ-rennes1.fr}

 \medskip

  \noindent\textsc{UMR 6625 CNRS Institut de Recherche Math\'ematique de
    Rennes (IRMAR) \\ Universit\'e de Rennes I, Campus de Beaulieu, F-35042
    Rennes \textsc{Cedex}, France.}

 \bigskip
  
 \noindent Florent \textsc{Malrieu}, corresponding author,
 e-mail: \texttt{florent.malrieu(AT)univ-rennes1.fr}

 \medskip

  \noindent\textsc{UMR 6625 CNRS Institut de Recherche Math\'ematique de
    Rennes (IRMAR) \\ Universit\'e de Rennes I, Campus de Beaulieu, F-35042
    Rennes \textsc{Cedex}, France.}

}

\end{document}